\documentclass[a4paper,12pt]{scrartcl}
\usepackage[english]{babel}
\usepackage{microtype}
\usepackage{amsmath,amssymb,amsthm,amsfonts,graphicx,etoolbox,scrlayer-scrpage}
\usepackage{bm}
\usepackage{mathtools}
\usepackage[noblocks]{authblk}
\usepackage{bbm}
\usepackage[toc,page]{appendix}
\usepackage{xcolor}
\usepackage{pagecolor}
\usepackage{hyperref}
\hypersetup{colorlinks=true, linkcolor=red, citecolor=cyan}
\usepackage[shortlabels]{enumitem}

\makeatletter
\patchcmd{\@maketitle}{\huge}{\Large}{}{}
\patchcmd{\abstract}{\quotation}{}{}{}
\AtBeginEnvironment{abstract}{\noindent\ignorespaces}
\AtEndEnvironment{abstract}{\par\mbox{}}

\newcommand{\shorttitle}{\@title}
\makeatother
\setkomafont{author}{\small}
\setkomafont{title}{\rmfamily\bfseries}
\setkomafont{disposition}{\rmfamily\bfseries}

\def\AMS#1{\par\noindent \textbf{AMS subject classification: }#1\par}
\newcommand{\acknowledgements}{\par\mbox{}\par\noindent\textbf{Acknowledgements: }}
\newcommand{\keywords}[1]{\par\noindent\textbf{Keywords: }#1}
\theoremstyle{plain}
\newtheorem{theorem}{Theorem}
\newtheorem{lemma}{Lemma}
\newtheorem{definition}{Definition}
\newtheorem{remark}{Remark}

\renewenvironment{abstract}{\bigskip\noindent\begin{minipage}{\textwidth}\setlength{\parindent}{15pt}\paragraph{Abstract:}}{\end{minipage}}

\usepackage{xcolor}
\definecolor{vs-darkm-black}{HTML}{1F1F1F}
\definecolor{vs-darkm-white}{HTML}{D0D0D0}

\usepackage{mlmodern}
\usepackage[T1]{fontenc}

\begin{document}

\title{Moment Estimator-Based Extreme Quantile Estimation with Erroneous
Observations: \\ Application to Elliptical Extreme Quantile Region Estimation}

\author[1]{Jaakko Pere\thanks{Corresponding author: jaakko.pere@aalto.com}}
\author[1]{Pauliina Ilmonen}
\author[2]{Lauri Viitasaari}
\affil[1]{Aalto University School of Science, Finland}
\affil[2]{Aalto University School of Business, Finland}

\maketitle

\begin{abstract}
    In many application areas of extreme value theory, the variables of interest
	are not directly observable but instead contain errors. In this article, we
	quantify the effect of these errors in moment-based extreme value index
	estimation, and in corresponding extreme quantile estimation. We consider
	all, short-, light-, and heavy-tailed distributions. In particular, we
	derive conditions under which the error is asymptotically negligible. As an
	application, we consider affine equivariant extreme quantile region
	estimation under multivariate elliptical distributions. 
\end{abstract}


\smallskip

% 62G32 = Statistics of extreme values; tail inference
% 62H12 = Estimation in multivariate analysis
\AMS{62G32, 62H12}

%%%%%%%%%%%%%%%%%%%%%%%%%%%%%%%%%%%%%%%%%%%%%%%%%%%%%%%%%%%%%%%%%%%%%%%%%%%%%%%%

\section{Introduction}
\label{sec:intro}

In extreme value theory it is customary to require that a normalized sample
maximum converges weakly to a limit distribution characterized, up to a location
and scale, with one parameter $\gamma\in\mathbb{R}$, called the extreme value
index. This allows to make inference about the tail of the distribution. For
further reading on extreme value theory and its applications,
see \cite{dehaan2007}.

When extreme behavior is considered in multivariate or infinite dimensional
settings, the problem is often reduced to univariate case. For example, under
multivariate elliptical distributions, the tail behavior is dictated by a
univariate variable, called the generating variate. Similarly, in functional
data setting, one may consider the extreme behavior of functionals such as
norms. In the elliptical case, one does not observe the generating variates
directly. On the other hand, approximated generating variates can be obtained by
plugging in estimators of location and scatter.

In this article, we quantify the effect of approximation errors in moment-based
estimation of the extreme value index $\gamma\in\mathbb{R}$, and in estimation
of the corresponding extreme quantiles. As an application, we consider affine
equivariant extreme quantile estimation under multivariate elliptical
distributions. On related work, we cite
\cite{cai2011,he2016,hult2002,kim2019,pere2024a,virta2024}. In \cite{cai2011},
the authors studied density-function based risk region estimation under
multivariate regular variation, whereas, in \cite{he2016}, the authors derived
depth-based extreme quantile regions under multivariate regular variation. In
\cite{hult2002}, the authors connected the extreme behavior of multivariate
elliptical distributions to the extreme behavior of the corresponding generating
variate. This connection was later applied in \cite{pere2024a}, where the
authors assessed the effect of approximation errors under heavy-tailed
elliptical distributions. In \cite{kim2019}, the authors considered
approximation errors in assessing the extreme behavior of PCA scores, again,
under heavy-tailed distributions. In \cite{virta2024}, the authors studied the
effect of approximation errors on the estimation of extreme value indices of
latent variables. We note that in most of these articles, results are derived
only under heavy-tailed distributions. Our work compliments the aforementioned
literature by considering short-, light-, and heavy-tailed distributions.

Our main result, Theorem~\ref{theorem:moment}, provides error bounds related to
extreme quantile estimation under approximated observations. Our result
highlights that extreme quantile estimation is very sensitive to approximation
errors, particularly in the case $\gamma\leq 0$, when decaying approximation
error does not automatically guarantee decay of the estimation error. On the
other hand, if the approximation error decays rapidly enough, then one obtains
standard convergence of the estimation error. Sufficient conditions for the rate
of decay of the approximation error are provided in
Lemma~\ref{lemma:error-behavior}. Theorem~\ref{theorem:moment} is applied in the
context of elliptical distributions. 

The rest of the article is organized as follows. We begin by presenting our
notations and by reviewing necessary preliminaries on univariate extreme value
theory in Section~\ref{sec:preliminaries}. Our main results are provided and
discussed in Section \ref{sec:general-error}. In Section~\ref{sec:elliptical},
we apply the results for affine equivariant extreme quantile region estimation
under multivariate elliptical distributions.   Proofs and technical lemmas are
postponed to the \hyperref[appendix]{Appendix}.

\section{Notations and Preliminaries}
\label{sec:preliminaries}

In this section, we present our notations and review necessary preliminaries on
univariate extreme value theory. For a general overview of the topic, we refer
to the textbook~\cite{dehaan2007}.

Let $Y$ be a random variable with a distribution $F_Y$. We use the notation
$U_Y(t) = F_Y^{\leftarrow}\left(1 - \frac{1}{t}\right)$ for $t > 1$, where
$g^\leftarrow(x) = \inf\left\{y\in\mathbb{R} : g(y)\geq x \right\}$ is the
left-continuous inverse of a nondecreasing function $g$. Note that the right
endpoint of the distribution $F_Y$ is given by  $U_Y\left(\infty\right) =
\lim_{t\to\infty} U_Y(t) = \sup\left\{x\in \mathbb{R} : F_Y(x) < 1\right\}$.
Whenever it is clear from the context, we omit the subscript and simply write
$F$, $U$, and $f=f_Y$ for the density (provided that it exists). Let $\bm Y =
\left(Y_1, \ldots, Y_n\right)$ be an i.i.d.\ sample of size $n$ from $F$, and
let $ \bm Y_{1,n}\leq  \bm Y_{2, n}\leq \cdots \leq  \bm Y_{n,n}$ be the
corresponding order statistics. Notations $\stackrel{\mathcal{D}}{=}$,
$\stackrel{\mathcal{D}}{\to}$, and $\stackrel{\mathbb{P}}{\to}$ denote equality
in distribution, convergence in distribution, and convergence in probability,
respectively. For two sequences of random variables $X_n$ and $Y_n$ we denote
$X_n = O_\mathbb{P}\left(Y_n\right)$ if there exists another uniformly tight
sequence of random variables $Z_n$ such that $X_n = Y_n Z_n$. Similarly, we
denote $X_n = o_\mathbb{P}\left(Y_n\right)$ if there exists a sequence of random
variables $Z_n$ converging to zero in probability such that $X_n = Y_n Z_n$. For
limits, if not specified explicitly, we assume $n\to\infty$. 

The distribution $F$ is said to be in the maximum domain of attraction of a
nondegenerate distribution $G$, denoted by $F\in\mathcal{D}\left(G\right)$, if
there exist sequences $a_n > 0$ and $b_n\in\mathbb{R}$ such that 
\begin{equation*}
	\frac{\bm Y_{n,n} - b_n}{a_n}\stackrel{\mathcal{D}}{\to} G.
\end{equation*}
The possible limit distributions $G$ are of the type
\begin{equation*}
	G_\gamma\left(x\right) =
	\begin{cases}
		\exp\left(-\left(1 + \gamma x\right)^{-1/\gamma}\right),\
		1 + \gamma x > 0 \quad &\textnormal{if}\ \gamma\neq 0, \\
		\exp\left(-e^{-x}\right),\
		x\in\mathbb{R} \quad &\textnormal{if}\ \gamma = 0.
	\end{cases}
\end{equation*}
The parameter $\gamma$ is called the extreme value index, and it characterizes
the heaviness of the tail of a distribution
$F\in\mathcal{D}\left(G_\gamma\right)$.

The function $U$ belongs to the class of extended regularly varying functions
with the extreme value index $\gamma\in\mathbb{R}$, and we denote $U\in
ERV_\gamma$, if 
\begin{equation} \label{eq:erv}
	\lim_{t\to\infty} \frac{U(tx) - U(t)}{a(t)} =
	\frac{x^\gamma - 1}{\gamma}	\quad \forall \, x > 0,
\end{equation}
where $a(t)=a_Y(t)$ is a suitable scaling function and for $\gamma = 0$ we
interpret $\frac{x^\gamma - 1}{\gamma} = \ln x$. Since we have
$F\in\mathcal{D}\left(G_\gamma\right)$ if and only if $U\in ERV_\gamma$, the
condition $U\in ERV_\gamma$ is convenient for constructing extreme quantile
estimators. Let $k<n$ denote a positive integer and suppose that we are
interested in estimating $(1-p)$-quantile $U\left(\frac{1}{p}\right)$, where $p$
is small compared to $k/n$. Then the $\left(1 - p\right)$-quantile is related to
the smaller $\left(1 - k/n\right)$-quantile by
\begin{equation*} 
	U\left(\frac{1}{p}\right) \approx U\left(\frac{n}{k}\right)
	+ a\left(\frac{n}{k}\right) \frac{\left(\frac{k}{np}\right)^\gamma - 1}
	{\gamma}.
\end{equation*}
The role of the integer $k$ is to control the number of tail observations used
in the estimation, and for asymptotic analysis, a standard assumption is 
\begin{enumerate}[start=1, label={(A\arabic*)}]
	\item $k = k_n\to\infty$, $k/n\to 0$, as $n\to\infty$. \label{item:k}
\end{enumerate}
Moreover, in asymptotics related to extreme quantile estimation, one usually
assumes that $p = p_n\to 0$ fast. See Remark~\ref{remark:p} for details. 

To estimate an extreme quantile $U\left(1/p\right)$ corresponding to a small
$p$, we have to estimate the different components $U\left(n/k\right)$,
$a\left(n/k\right)$ and $\gamma$. To cover all the cases $\gamma\in \mathbb{R}$,
we use the moment-based estimator. Set, for  $\ell\in\{1, 2\}$,
\begin{equation*}
	M_n^{(\ell)}\left(\bm Y\right) = \frac{1}{k}\sum_{j=0}^{k-1}
	\left(\ln \bm Y_{n-j, n} - \ln \bm Y_{n-k, n}\right)^\ell,
\end{equation*}
and
\begin{equation*}
	\hat\gamma_+\left(\bm Y\right) = M_n^{(1)}\left(\bm Y\right)
	\quad\textnormal{and}\quad
	\hat\gamma_-\left(\bm Y\right) = 1
	- \frac{1}{2}\left(1 - \frac{\left(M_n^{(1)}\left(\bm Y\right)\right)^2}
	{M_n^{(2)}\left(\bm Y\right)}\right)^{-1}.
\end{equation*}
It is known that $\hat\gamma_+\left(\bm Y\right)\stackrel{\mathbb{P}}{\to}
\gamma_+ =\max\{0, \gamma\}$ and $\hat\gamma_-\left(\bm
Y\right)\stackrel{\mathbb{P}}{\to} \gamma_- =\min\{0, \gamma\}$ provided that
$U\left(\infty\right) > 0$, $U\in ERV_\gamma$, and Assumption~\ref{item:k} holds
(see, e.g., \cite[Lemma 3.5.1]{dehaan2007}). Note that $\hat\gamma_+\left(\bm
Y\right)$ gives the well-studied Hill estimator~\cite{hill1975} that is often
used in the case $\gamma>0$. The moment estimator, introduced
in~\cite{dekkers1989}, is given by 
\begin{equation*}
	\hat\gamma_M\left(\bm Y\right) = \hat\gamma_+\left(\bm Y\right)
	+ \hat\gamma_-\left(\bm Y\right).
\end{equation*}
The corresponding extreme quantile estimator is defined as 
\begin{equation} \label{eq:q-uni-est}
	\hat x_p\left(\bm Y\right) = \bm Y_{n-k,n}
	+ \hat \sigma_M\left(\bm Y\right) \frac{\left(\frac{k}{np}\right)
		^{\hat\gamma_M\left(\bm Y\right)} - 1}{\hat \gamma_M\left(\bm Y\right)},
\end{equation}
where $\hat\sigma_M\left(\bm Y\right) = \bm Y_{n-k, n} M_n^{(1)}\left(\bm
Y\right)\left(1 - \hat\gamma_-\left(\bm Y\right)\right)$~\cite{dekkers1989}.

For asymptotic normality results of $\bm Y_{n-k,n}$, $\hat\gamma_M\left(\bm
Y\right)$, $\hat\sigma_M\left(\bm Y\right)$, and $\hat x_p\left(\bm Y\right)$,
the condition $U\in ERV_\gamma$ is not sufficient but a second-order condition
is required. The function $U$ satisfies the second-order extended regular
variation condition with $\gamma\in\mathbb{R}$ and $\rho \leq 0$  if for some
function $A$, that has constant sign and $\lim_{t\to\infty}A\left(t\right) = 0$,
we have
\begin{equation} \label{eq:2erv}
	\lim_{t\to\infty} \frac{\frac{U\left(tx\right) - U\left(t\right)}
		{a\left(t\right)} - \frac{x^\gamma - 1}{\gamma}}{A(t)}
	= \frac{1}{\rho}\left(\frac{x^{\gamma + \rho} - 1}{\gamma + \rho}
	- \frac{x^\gamma - 1}{\gamma}\right)
	\eqqcolon H_{\gamma, \rho}\left(x\right) \quad\forall \, x > 0,
\end{equation}
where for $\gamma = 0$ or $\rho = 0$, the right-hand side is interpreted as the
limit of $H_{\gamma, \rho}\left(x\right)$ as $\gamma\to 0$ or $\rho\to 0$. If
$U$ satisfies the second-order extended regular variation condition, we write
$U\in 2ERV_{\gamma, \rho}$. The second-order condition gives the rate of
convergence in~\eqref{eq:erv}. For the asymptotic normality of $\bm Y_{n-k,n}$,
the condition $U\in 2ERV_{\gamma, \rho}$ with $\gamma\in\mathbb{R}$ and
$\rho\leq 0$ is sufficient provided that $k$ is chosen such that
Assumption~\ref{item:k} holds and $\lim_{n\to\infty}\sqrt{k} A\left(n/k\right)$
exists and is finite, see~\cite[Theorem 2.4.1]{dehaan2007}. On the other hand,
for asymptotic normality of $\hat\gamma_M\left(\bm Y\right)$ and
$\hat\sigma_M\left(\bm Y\right)$, second-order condition for the function $\ln
U$ is necessary, see~\cite[Theorem 3.5.4 and Theorem 4.2.1]{dehaan2007}. The
following result connects second-order conditions of $U$ and $\ln U$.
\begin{lemma}[Lemma B.3.16, \cite{dehaan2007}] \label{lemma:log-second} Suppose
	that $U\in 2ERV_{\gamma, \rho}$ with $U\left(\infty\right) \in (0, \infty]$
	and suppose that $\gamma \neq \rho$. Then
	\begin{equation*}
		\lim_{t\to\infty} \frac{\frac{a\left(t\right)}{U\left(t\right)}
		- \gamma_+}{A\left(t\right)} =
		\begin{cases}
			0, & \gamma < \rho \leq 0, \\
			\pm\infty,
			& \rho < \gamma \leq 0 \ \textnormal{or}\
			\left(0 < \gamma < -\rho \ \textnormal{and}\ l \neq 0\right)
			\ \textnormal{or}\ \gamma = -\rho, \\
			\frac{\gamma}{\gamma + \rho},
			& \left(0 < \gamma < -\rho \ \textnormal{and}\ l = 0\right)
			\ \textnormal{or}\ \gamma > -\rho \geq 0,
		\end{cases}
	\end{equation*}
	where $l = \lim_{t\to\infty}
	U\left(t\right) - a\left(t\right) / \gamma$ for $\gamma>0$. 
	Furthermore, if $\gamma > 0$ and $\rho < 0$, we have
	\begin{equation} \label{eq:2erv-log}
		\lim_{t\to\infty} \frac{\frac{\ln U\left(tx\right)
		- \ln U\left(t\right)}{a\left(t\right) / U\left(t\right)}
		- \frac{x^{\gamma_-} - 1}{\gamma_-}}{Q\left(t\right)}
		= H_{\gamma_-, \rho'}\left(x\right),
	\end{equation}
	where
	\begin{equation} \label{eq:Q}
		Q\left(t\right) =
		\begin{cases}
			A\left(t\right), & \gamma < \rho \leq 0, \\
			\gamma_+ - \frac{a\left(t\right)}{U\left(t\right)},
			& \rho < \gamma \leq 0 \ \textnormal{or}\
			\left(0 < \gamma < -\rho \ \textnormal{and}\ l \neq 0\right)
			\ \textnormal{or}\ \gamma = -\rho, \\
			\frac{\rho}{\gamma + \rho} A\left(t\right),
			& \left(0 < \gamma < -\rho \ \textnormal{and}\ l = 0\right)
			\ \textnormal{or}\ \gamma > -\rho > 0,
		\end{cases}
	\end{equation}
	and
	\begin{equation*}
		\rho' =
		\begin{cases}
			\rho, & \gamma < \rho \leq 0, \\
			\gamma, & \rho < \gamma \leq 0, \\
			-\gamma, & \left(0 < \gamma < -\rho
			\ \textnormal{and}\ l \neq 0\right), \\
			\rho, & \left(0 < \gamma < -\rho \ \textnormal{and}\ l = 0\right)
			\ \textnormal{or}\ \gamma \geq -\rho > 0.
		\end{cases}
	\end{equation*}
	Finally, if $\gamma > 0$ and $\rho = 0$, the limit in~\eqref{eq:2erv-log}
	equals zero for any $Q(t)$ satisfying $A(t) = O\left(Q(t)\right)$.
\end{lemma}

\section{Main Results}
\label{sec:general-error}
We assess next the effect of replacing true variables with erroneous
approximations in the estimation of $\gamma=\gamma_Y$,
$U\left(n/k\right)=U_Y\left(n/k\right)$,
$a\left(n/k\right)=a_Y\left(n/k\right)$, and
$U\left(1/p\right)=U_Y\left(1/p\right)$. We restrict to almost surely positive
random variables since this is sufficient for our applications. Note that this
further implies $U\left(\infty\right) > 0$ which is needed for the estimators to
be well-defined. This assumption is usually assumed implicitly in the context of
extreme value theory.

\begin{theorem}
	\label{theorem:moment}
	Let $Y$ be an almost surely positive random variable with $F\in
	\mathcal{D}\left(G_\gamma\right)$, $\gamma\in\mathbb{R}$. Let $\bm Y =
	\left(Y_1, \ldots, Y_n\right)$ be i.i.d. copies of $Y$ and let $\hat{\bm Y}
	= \left(\hat Y_1, \ldots, \hat Y_n\right)$ be an approximated sample.
	Suppose Assumption \ref{item:k} holds and that
		\begin{equation} \label{eq:error-cond}
		\max_{0\leq j\leq k} \left|\frac{\hat{\bm Y}_{n-j,n}}
		{\bm Y_{n-j,n}} - 1\right|
		= O_\mathbb{P}\left(h_n\right)
	\end{equation}
	for some sequence $h_n$ such that  $z_n =
	\frac{h_nU\left(\frac{n}{k}\right)}{a\left(\frac{n}{k}\right)} = o(1)$.
	Then
	\begin{align}
		\left|\hat\gamma_M\left(\hat{\bm Y}\right)
		- \hat\gamma_M\left(\bm Y\right)\right|
		&= O_\mathbb{P}\left(z_n\right), \label{eq:moment-error} \\
		\left|\frac{\hat{\bm Y}_{n-k,n} - \bm Y_{n-k,n}}{a(n/k)}\right|
		&=  O_\mathbb{P}\left(z_n\right), \quad \textnormal{and}
		\label{eq:order-stat-error} \\
		\left|\frac{\hat\sigma_M\left(\hat{\bm Y}\right)
		- \hat\sigma_M\left(\bm Y\right)}{a(n/k)}\right|
		&= O_\mathbb{P}\left(z_n\right). \label{eq:a-error}
	\end{align}
	Furthermore, suppose that 
	\begin{enumerate}[start=2, label={(A\arabic*)}]
		\item $p = p_n$, $np = o\left(k\right)$ and $\ln\left(np\right) =
		o\left(\sqrt{k}\right)$ \label{item:p-usual-1}
		\end{enumerate}
	and
		\begin{enumerate}[start=3, label={(A\arabic*)}]
			\item $\sqrt{k}\left(\hat\gamma_M\left(\bm Y\right) - \gamma\right) =
		O_\mathbb{P}\left(1\right)$ and $\sqrt{k}z_n = O\left(1\right)$
		\label{item:gamma-bounded}
	\end{enumerate}
	hold. Let $d_n = k/(np)$ and $q_\gamma\left(t\right) = \int_1^t s^{\gamma -
	1}\ln s \, \mathrm{d}s$. Then
	\begin{equation}
		\left|\frac{\hat x_{p}\left(\hat{\bm Y}\right)
		- \hat x_{p}\left(\bm Y\right)}{a\left(\frac{n}{k}\right)
		q_\gamma\left(d_n\right)}\right|
		= O_\mathbb{P}\left(z_n\right). \label{eq:q-error}
	\end{equation}
\end{theorem}

\begin{remark}
	\label{remark:normality}
	In order to obtain asymptotic normality, let $U\in 2ERV_{\gamma, \rho}$ with
	$\gamma\neq\rho$. Then, under suitable additional
	conditions\footnote{Precise conditions are $\gamma\neq \rho$ and
	$\sqrt{k}Q\left(n/k\right)\to\lambda\in\mathbb{R}$ with $Q = A$
	from~\eqref{eq:2erv} if $\gamma > 0$ and $\rho = 0$, and $Q$ from
	\eqref{eq:Q} otherwise.}
	\begin{equation} \label{eq:standard-normality-components}
		\sqrt{k}\left(\hat\gamma_M\left(\bm Y\right) - \gamma,
		\frac{\sigma_M\left(\bm Y\right)}{a\left(\frac{n}{k}\right)} - 1,
		\frac{\bm Y_{n-k,n}
			- U\left(\frac{n}{k}\right)}{a\left(\frac{n}{k}\right)}\right)
		\stackrel{\mathcal{D}}{\to} \left(\Gamma, \Lambda, B\right)
	\end{equation}
	for jointly normal $(\Gamma,\Lambda,B)$, see~\cite[Corollary
	4.2.2]{dehaan2007}. For the extreme quantile estimator, this translates into 
	\begin{equation} \label{eq:standard-normality-q}
		\sqrt{k}\frac{\hat x_{p}\left(\bm Y\right)
		- U\left(1/p\right)}{a\left(\frac{n}{k}\right)
		q_\gamma\left(d_n\right)} \stackrel{\mathcal{D}}{\to}
		\Gamma + \left(\gamma_-\right)^2 B - \gamma_-\Lambda
		- \lambda\mathbbm{1}_{\{\gamma < \rho \leq 0\}}\frac{\gamma_-}
		{\gamma_- + \rho},
	\end{equation}
	where the coefficient $\lambda$ related to the bias term can be derived from
	Lemma~\ref{lemma:log-second}, see~\cite[Theorem 4.3.1]{dehaan2007}. Now, if
	$\sqrt{k}z_n \to 0$, Slutsky's lemma implies that we can replace $\bm Y$ in
	limits~\eqref{eq:standard-normality-components}
	and~\eqref{eq:standard-normality-q} with $\hat{\bm Y}$.
\end{remark}

In order to obtain rate of convergence and limiting distribution, it is required
that $\sqrt{k} z_n\to 0$. The following lemma provides sufficient simpler
conditions for this.
\begin{lemma}
	\label{lemma:error-behavior}
	Suppose~\ref{item:k} and that $U\in 2ERV_{\gamma, \rho}$ with $\rho < 0$ and
	$\gamma\in\mathbb{R}$. The sequence $z_n = \frac{h_n
	U\left(n/k\right)}{a\left(n/k\right)}$ satisfies $\sqrt{k} z_n\to 0$, if one
	of the following conditions holds:
	\begin{enumerate}
		\item $\gamma > 0$ and $\sqrt{k} h_n \to 0$, \label{item:gamma-pos}
		
		\item $\gamma = 0$ and $\sqrt{k} h_n =
		O\left(\left(\frac{n}{k}\right)^{-\delta}\right)$ for some $\delta > 0$,
		\label{item:gamma-0}

		\item $\gamma < 0$ and $\sqrt{k} h_n =
		o\left(\left(\frac{n}{k}\right)^{\gamma}\right)$.
		\label{item:gamma-neg}
	\end{enumerate}
\end{lemma}

\section{Extreme Quantile Region Estimation for Multivariate Elliptical
Distributions}
\label{sec:elliptical}

Let $X$ be a multivariate random variable with a density $f_X$ and consider
quantile regions of the form
\begin{equation}
	\label{eq:region}
	Q_p = \{x\in \mathbb{R}^d : f_X(x) \leq \beta_p \},
\end{equation}
where $\beta_p$ is chosen such that $\mathbb{P}\left(X\in Q_p\right) = p$.
Regions $Q_p$ corresponding to a small probability $p$ are called extreme
quantile regions. In this section we apply Theorem~\ref{theorem:moment} to
construct an estimator for the extreme quantile region $Q_p$ under the
assumption of multivariate ellipticity. In
Section~\ref{sec:elliptical-preliminaries} we review elliptical distributions
and in Section~\ref{sec:elliptical-main-result} we study the estimation
procedure. These results complement \cite{pere2024a} in which only the
heavy-tailed case and the Hill estimator was analyzed. However, we note that
Theorem \ref{theorem:moment} could be applied in various other settings as well.

\subsection{Elliptical Distribution and Assumptions on the Generating Variate}
\label{sec:elliptical-preliminaries}
The following gives definition of multivariate elliptical distributions. Note
that one could give the definition in terms of characteristic functions,
see~\cite[Definition 2]{frahm2004}. However, using \cite[Theorem
1]{cambanis1981} leads to the following convenient stochastic representation. 

\begin{definition}
	\label{def:ellipticity}
	Let $\mu\in\mathbb{R}^d$ be a vector and $\Sigma\in\mathbb{R}^{d\times d}$ a
	symmetric positive definite matrix with $\det\left(\Sigma\right) = 1$. Let
	$\mathcal{R}$ a nonnegative random variable and $S$ a $d$-variate random
	vector uniformly distributed over the unit-sphere $\{ x\in\mathbb{R}^d:
	x^\intercal x = 1\}$ such that $\mathcal{R}$ and $S$ are independent. A
	$d$-variate random variable $X$ is elliptically distributed with the
	location vector $\mu$, the scatter matrix $\Sigma$ and the generating
	variate $\mathcal{R}$ if
	\begin{equation} \label{eq:elliptical-representation}
	  X \stackrel{\mathcal{D}}{=} \mu + \mathcal{R} \Sigma^{1/2} S,
	\end{equation}
	where $\Sigma^{1/2} \in \mathbb{R}^{d\times d}$ is the unique symmetric
	positive definite matrix such that $\Sigma = \Sigma^{1/2} \Sigma^{1/2}$.
\end{definition}

Note that while the matrix $\Sigma^{1/2}$ could be replaced by any other matrix
$\Lambda$ such that $\Lambda \Lambda^\intercal = \Sigma$ and
$\mathrm{Rank}\left(\Lambda\right) = \mathrm{Rank}\left(\Sigma\right)$, this
choice does not play role in what follows and hence we take the symmetric square
root $\Sigma^{1/2}$. Moreover, without the assumption $\det\left(\Sigma\right) =
1$, the generating variate $\mathcal{R}$ and the scatter matrix $\Sigma$ are
unique only up to a positive constant. That is, the parameters $\left(\mu,
\mathcal{R}, \Sigma\right)$ and $\left(\mu, \frac{1}{\sqrt{c}}\mathcal{R},
c\Sigma\right)$, $c > 0$, define the same model. To guarantee identifiability,
typical constraints include  $\det\left(\Sigma\right) = 1$, $\Sigma_{11} = 1$,
or $\mathrm{Trace}\left(\Sigma\right) = d$~\cite{paindaveine2008}. 

In the sequel, we assume that the generating variate is absolutely continuous
with density $f_\mathcal{R}$, implying that the density of the corresponding
elliptically distributed random variable $X$ exists, see~\cite[Corollary
4]{frahm2004}. Note that this implies  $U_\mathcal{R}\left(\infty\right)\in (0,
\infty]$. In addition, we pose the following assumptions on the generating
variate $\mathcal{R}$. 
\begin{enumerate}[start=4, label={(A\arabic*)}]
	\item $\mathcal{R}$ is supported on $(0,U_\mathcal{R}(\infty))$ and the
	density $f_\mathcal{R}$ is eventually decreasing.\label{item:density}
	\item $U_\mathcal{R} \in 2ERV_{\gamma, \rho}$ with $\gamma > -1/2$ and $\rho
	< 0$ such that $\gamma\neq\rho$. \label{item:2erv}
\end{enumerate}
The first assumption on the support simply states that there are no intervals
inside $(0,U_\mathcal{R}(\infty))$ with zero probability. Together with
$f_\mathcal{R}$ being eventually decreasing this implies that, for a
sufficiently small $p$, we can represent quantile regions defined
in~\eqref{eq:region} by
\begin{equation} \label{eq:q-representation}
	Q_p = \left\{x \in\mathbb{R}^d : \|x - \mu\|_\Sigma \geq r_p\right\},
\end{equation}
where $r_p = U_\mathcal{R}\left(1/p\right)$ and $\|u\|_A = \sqrt{u^\intercal
A^{-1} u}$ denotes the norm induced by a symmetric positive definite matrix
$A\in\mathbb{R}^{d\times d}$ (see~\cite[Section 3.2, pages 5--6]{pere2024a}).

\subsection{Construction of the Estimator and Consistency}
\label{sec:elliptical-main-result}

Let $X$ be a $d$-variate elliptically distributed random variable with the
location vector $\mu$, the scatter matrix $\Sigma$, and the generating variate
$\mathcal{R}$. The idea is to estimate $\mu$, $\Sigma$, and $r_p$ in
\eqref{eq:q-representation}. Let $\bm X = \left(X_1, \ldots, X_n\right)$ be an
i.i.d.\ sample from $X$. Assume that $\hat\mu\left(\bm X\right)$ and
$\hat\Sigma\left(\bm X\right)$ are $\sqrt{n}$-consistent estimators of the
location vector $\mu$ and the scatter matrix $\Sigma$. Throughout, we also
assume that $\hat\Sigma\left(\bm X\right)$ is symmetric positive definite.
Denote
\begin{align}
	R &= \|X - \mu\|_{\Sigma}, &
	\hat R &= \|X - \hat\mu\left(\bm X\right)\|_{\hat\Sigma\left(\bm X\right)},
	\label{eq:r} \\
	R_i &= \|X_i - \mu\|_{\Sigma}, &
	\hat R_i &= \|X_i - \hat\mu\left(\bm X\right)\|
	_{\hat\Sigma\left(\bm X\right)} \label{eq:ri}.
\end{align}
Now, by \eqref{eq:elliptical-representation}, $R \stackrel{\mathcal{D}}{=}
\mathcal{R}$. Let $\bm R = \left(R_1, \ldots, R_n\right)$ and $\hat{\bm R} =
\left(\hat R_1, \ldots, \hat R_n\right)$. As we do not observe $\bm R$, the
estimator for the extreme quantile region $Q_p$ is defined as 
\begin{equation} \label{eq:q-estimator}
	\hat Q_p = \left\{x\in\mathbb{R}^d :
	\left\|x - \hat\mu\left(\bm X\right)\right\|_{\hat\Sigma\left(\bm X\right)}
	\geq \hat x_p\left(\hat{\bm R}\right)\right\},
\end{equation} 
where $\hat x_p\left(\hat{\bm R}\right)$ is the extreme quantile estimator for
the generating variate, given by \eqref{eq:q-uni-est}.

\begin{remark} \label{remark:loc-scat} It is well-known that if
	$\mathbb{E}\left(\mathcal{R}^4\right) < \infty$, then the sample mean vector
	and the sample covariance matrix are $\sqrt{n}$-consistent estimators for
	the location and for the scatter up to a positive constant. However, in the
	case $F_\mathcal{R}\in\mathcal{D}\left(G_\gamma\right)$ with $\gamma > 1/4$,
	the fourth moment is not finite~\cite[Exercise 1.16.]{dehaan2007}. Also, the
	sample mean and the sample covariance are extremely sensitive to outliers.
	However, even in the heavy-tailed case (large $\gamma$),
	$\sqrt{n}$-consistent and robust alternatives exist for the estimators of
	the location and the scatter up to a positive constant. We mention the
	minimum covariance determinant method (MCD), $M$-estimators, $R$-estimators,
	and $S$-estimators (see, e.g.,~\cite{frahm2009,hubert2018,paindaveine2014}
	and the references therein). Often in practice, as in our application, the
	estimation of $\Sigma$ up to a positive constant is sufficient. However,
	estimation of the specific scatter is usually feasible. For example, if
	$\sqrt{n}\left(\tilde\Sigma\left(\bm X\right)- c\Sigma\right)$ admits an
	asymptotic distribution for $c > 0$, then the asymptotic distribution of
	$\sqrt{n}\left(\hat\Sigma\left(\bm X\right)- \Sigma\right)$ can be acquired
	with the delta method. In our case we would set $\hat\Sigma\left(\bm
	X\right) = \frac{\tilde \Sigma\left(\bm X\right)}{\left(\det\left(\tilde
	\Sigma\left(\bm X\right)\right)\right)^{1/d}}$. 
\end{remark}

\begin{theorem} \label{theorem:Q-hat-consistent} Let $X$ be a $d$-variate
	elliptically distributed random variable with the location vector $\mu$, the
	scatter matrix $\Sigma$, and the generating variate $\mathcal{R}$ satisfying
	Assumptions~\ref{item:density} and~\ref{item:2erv}. Let $\bm X = \left(X_1,
	\ldots, X_n\right)$ be an i.i.d.\ sample from $X$. Denote
	$q_\gamma\left(t\right) = \int_1^t s^{\gamma - 1}\ln s \, \mathrm{d}s$ and
	$d_n = k/(np)$. Assume that $k=k_n$ satisfies \ref{item:k}, and assume
	that, as $n\to\infty$, the following conditions hold:
	\begin{enumerate}[start=6, label={(A\arabic*)}]
		\item $\sqrt{k} Q\left(n/k\right)\to\lambda\in\mathbb{R}$, where $Q$ is
		the auxiliary function from the second-order condition for $\ln
		U_\mathcal{R}$.
		\label{item:second}
		\item $np=np_n = O\left(1\right)$ and
		$q_\gamma\left(d_n\right)/\left(d_n^\gamma\sqrt{k}\right)\to 0$.
		\label{item:p-usual-2}
		\item $\sqrt{\frac{k}{n}}\frac{d_n^\gamma - 1}{\gamma} = O(1)$ (in the
		case $\gamma = 0$ we interpret $\frac{d_n^\gamma - 1}{\gamma} = \ln
		d_n$).
		\label{item:p-slow}
		\item $\sqrt{n}\left(\hat\mu\left(\bm X\right) - \mu\right) =
		O_\mathbb{P}\left(1\right)$ and $\sqrt{n}\left(\hat\Sigma\left(\bm
		X\right) - \Sigma\right) = O_\mathbb{P}\left(1\right)$.
		\label{item:loc-scat}
	\end{enumerate}
	Then, as $n\to\infty$, we have
	\begin{equation*} 
		\frac{\mathbb{P}\left(X \in \hat{Q}_p\triangle Q_p\right)}{p}
		\stackrel{\mathbb{P}}{\to} 0,
	\end{equation*}
	where $Q_p = Q_{p_n}$ is given by~\eqref{eq:region}, $\hat Q_p = \hat
	Q_{p_n}$ is given by~\eqref{eq:q-estimator}, and where $A\triangle B =
	\left(A\setminus B\right) \cup \left(B\setminus A\right)$ denotes the
	symmetric difference between the sets $A,B\subset\mathbb{R}^d$.
\end{theorem}
Note that in extreme quantile region estimation, it is customary to assess
consistency using symmetric differences,
see~\cite{cai2011,einmahl2013,einmahl2009,he2016}.

The following two remarks discuss the assumptions made in
Theorem~\ref{theorem:Q-hat-consistent}.

\begin{remark}
	In Theorem \ref{theorem:Q-hat-consistent} we suppose $\gamma > -1/2$. We
	note that Assumption~\ref{item:loc-scat} implies that we can set $h_n =
	1/\sqrt{n}$ in \eqref{eq:error-cond}, see Lemma~\ref{lemma:gvariate-error}
	in Appendix \ref{appendix:sec:applications}. Then  $\sqrt{k} z_n =
	\sqrt{\frac{k}{n}}
	\frac{U_{\mathcal{R}}\left(n/k\right)}{a_{\mathcal{R}}\left(n/k\right)}\to
	0$ provided that $U_{\mathcal{R}}\in ERV_{\gamma, \rho}$ with $\gamma >
	-1/2$ and $\rho < 0$, see Lemma~\ref{lemma:error-behavior}. However, in the
	case $\gamma \leq -1/2$, Assumption~\ref{item:loc-scat} does not guarantee
	that $\sqrt{k} z_n\to 0$. On the other hand, for $\gamma\leq -1/2$ one can
	rely on simpler procedure. Indeed, suppose $\mathcal{R}\in 2ERV_{\gamma,
	\rho}$ with $\gamma, \rho < 0$ and let $\bm R = (R_1, \ldots, R_n)$ be an
	i.i.d.\ sample of $\mathcal{R}$. Let $p = 1/(cn)$ for some $c\in(0, \infty)$
	and suppose that the assumptions of \cite[Theorem 4.3.1]{dehaan2007} hold.
	Then by combining Lemma~\ref{lemma:neves} and~\cite[Exercise
	1.15]{dehaan2007} we obtain
	\begin{equation*}
		n^{\left|\gamma\right|}\left(\bm R_{n,n}
		- U_{\mathcal{R}}\left(\frac{1}{p}\right)\right)
		= O_\mathbb{P}\left(1\right).
	\end{equation*}
	On the other hand, by combining Lemma~\ref{lemma:neves}, \cite[Theorem
	4.3.1]{dehaan2007}, and \eqref{eq:q} we have
	\begin{equation*}
		n^{\left|\gamma\right|} k^{1/2 - \left|\gamma\right|}
		\left(\hat x_p\left(\bm R\right)
		- U_{\mathcal{R}}\left(\frac{1}{p}\right)\right)
		= O_\mathbb{P}\left(1\right).
	\end{equation*}
	This suggests to use $\{x\in \mathbb{R}^d : \left\|x - \hat\mu\left(\bm
	X\right)\right\|_{\hat\Sigma\left(\bm X\right)} \geq \hat{\bm R}_{n,n}\}$
	instead of $\hat Q_p$ in the case $\gamma \leq -1/2$. 
\end{remark}

\begin{remark}
	\label{remark:p}
	Assumptions of the type~\ref{item:p-usual-1} or~\ref{item:p-usual-2} are
	standard in asymptotic results involving extreme quantile estimators. In
	these, the first part ($np = o(k)$ or $np = O(1)$) means that $p\to 0$ fast,
	while the second part ($\ln\left(np\right) = o\left(\sqrt{k}\right)$ or
	$q_\gamma\left(d_n\right)/\left(d_n^\gamma\sqrt{k}\right)\to 0$) means that
	$p\to 0$ cannot happen too fast, compared to the rate $\sqrt{k}$. Note also
	that \ref{item:p-usual-2} implies~\ref{item:p-usual-1}. Note also that
	$q_\gamma\left(d_n\right)/\left(d_n^\gamma\sqrt{k}\right)\to 0$ can be
	expressed  as 
	\begin{equation*}
		\begin{cases}
			\frac{\ln d_n}{\sqrt{k}} \to 0, & \gamma > 0, \\
			\frac{\left(\ln d_n\right)^2}{\sqrt{k}} \to 0, & \gamma = 0, \\
			d_n^\gamma\sqrt{k}\to\infty, & \gamma < 0,
		\end{cases}
	\end{equation*}
	since 
		\begin{equation} \label{eq:q}
			q_\gamma\left(t\right) \sim
			\begin{cases}
				\frac{1}{\gamma} t^\gamma \ln t, & \gamma > 0, \\
				\frac{1}{2} \left(\ln t\right)^2, & \gamma = 0, \\
				1 / \gamma^2, & \gamma < 0,
			\end{cases}
	\end{equation}
	as $t\to\infty$.

	Assumption~\ref{item:p-slow} ties the convergence rate $\sqrt{n}$ of the
	location and scatter estimators to the rate at which $p$ decays to zero.
	Again, $p\to 0$ cannot happen too fast, compared to the rate $\sqrt{n}$.
	Condition \ref{item:p-slow} is equivalent to
	\begin{equation*}
		\begin{cases}
			\sqrt{\frac{k}{n}} d_n^\gamma = O(1), & \gamma > 0, \\
			\sqrt{\frac{k}{n}} \ln d_n = O(1), & \gamma = 0, \\
			\sqrt{\frac{k}{n}} = O(1), & \gamma < 0.
		\end{cases}
	\end{equation*}
	This means that, in the case $\gamma < 0$, \ref{item:p-slow} follows from
	\ref{item:k}. In the case $\gamma \geq 0$ one has to assume
	both~\ref{item:p-usual-2} and~\ref{item:p-slow}. 
	
	Finally, note that all our results would remain valid with obvious
	modifications if the rate $\sqrt{n}$ of the location and scatter estimators
	in \ref{item:loc-scat} is replaced with any other rate $n^{\alpha}$.
\end{remark}

The next remark states affine equivariance of the estimator $\hat Q_p$, provided
that the corresponding estimators for the location and scatter are affine
equivariant. This fact can be proven similarly as~\cite[Theorem 6]{pere2024a}.
\begin{remark}
	\label{remark:affine}
	Let $A\in\mathbb{R}^{d\times d}$ be an invertible matrix and
	$a\in\mathbb{R}^d$. Let $\bm X = \left(X_1, \ldots, X_n\right)$ be sample of
	$X$, $Z_i = A X_i + a$, and $\bm Z = \left(Z_1, \ldots, Z_n\right)$. If
	$\hat\mu$ and $\hat\Sigma$ are affine equivariant estimators of $\mu$ and
	$\Sigma$, that is,
	\begin{equation*}
		\hat\mu\left(\bm Z\right) = A\hat\mu\left(\bm X\right) + a \quad
		\textnormal{and}\quad
		\hat\Sigma\left(\bm Z\right) = A\hat\Sigma\left(\bm X\right)A^\intercal,
	\end{equation*}
	then 
	\begin{equation*}
		\hat Q_p' = \left\{Ax + b: x\in Q_p\right\},
	\end{equation*}
	where 
	\begin{align*}
		\hat R_i' &= \left\|Z_i - \hat\mu\left(\bm Z\right)\right\|
		_{\hat\Sigma\left(\bm Z\right)}, \\
		\hat{\bm{R}'} &= \left(\hat R_i', \ldots, \hat R_n'\right)
		\quad\textnormal{and} \\
		\hat Q_p' &= \left\{z\in\mathbb{R}^d :
		\left\|z - \hat\mu\left(\bm Z\right)\right\|
		_{\hat\Sigma\left(\bm Z\right)} \geq \hat x_p\left(\hat{\bm R'}\right)
		\right\}.
	\end{align*}
\end{remark}

\bigskip
\acknowledgements{Jaakko Pere gratefully acknowledges support from the Vilho,
Yrj\"o and Kalle V\"ais\"al\"a Foundation. Pauliina Ilmonen gratefully
acknowledges support from the Academy of Finland via the Centre of Excellence in
Randomness and Structures, decision number 346308.}

\bibliographystyle{abbrv}
\bibliography{sources}

\begin{appendices}
	\label{appendix}

	\section{Proofs of Section \ref{sec:general-error}}
	\label{appendix:sec:evt}

	Before proving Theorem~\ref{theorem:moment} and
	Lemma~\ref{lemma:error-behavior}, we review two lemmas that describe the
	behaviors of the functions $U$ and $a$ under the conditions $U\in
	ERV_\gamma$ and $U\in 2ERV_{\gamma, \rho}$ for $\gamma\in\mathbb{R}$ and
	$\rho < 0$. Recall that the function $U$ corresponding to a
	distribution $F$ is slowly varying if
	\begin{equation*}
		\lim_{t\to\infty} \frac{U\left(tx\right)}{U\left(t\right)} = 1
		\quad\forall \, x > 0.
	\end{equation*}

	\begin{lemma}[\cite{dehaan2007}, Lemma 1.2.9]
		\label{lemma:dehaan}
		Suppose $U\in ERV_\gamma$.
		\begin{enumerate}
			\item If $\gamma > 0$, then $U\left(\infty\right) =
			\infty$ and $\lim_{t\to\infty} U\left(t\right) / a\left(t\right) = 1
			/ \gamma$. \label{item:gamma-pos-dehaan}

			\item If $\gamma = 0$, then $U$ is slowly varying and
			$\lim_{t\to\infty} a\left(t\right) / U\left(t\right) = 0$.
			\label{item:gamma-0-dehaan}

			\item If $\gamma < 0$, then $U\left(\infty\right) < 0$ and
			$\lim_{t\to\infty} a\left(t\right) = 0$.
			\label{item:gamma-neg-dehaan}
		\end{enumerate}
	\end{lemma}

	\begin{lemma}[\cite{neves2009}, Lemma 3] \label{lemma:neves}
		Suppose $U\in 2ERV_{\gamma, \rho}$ with $\gamma\in\mathbb{R}$ and $\rho
		< 0$. Then $\lim_{t\to\infty} t^{-\gamma} a(t) = c\in (0, \infty)$.
	\end{lemma}

	We next prove our main result, Theorem \ref{theorem:moment}.
	\begin{proof}[Proof of Theorem \ref{theorem:moment}]
		Throughout the proof, let $U = U_Y$ and $a = a_Y$. First, let us show
		that~\eqref{eq:moment-error} holds. Similarly as in the proof
		of~\cite[Theorem 2]{virta2024}, one obtains
		\begin{align}
			\left|\hat\gamma_+\left(\hat{\bm Y}\right) - \hat\gamma_+
			\left(\bm Y\right)\right|
			&= O_\mathbb{P}\left(h_n\right), \label{eq:hill-error-rate} \\
			\left|\frac{\left(M_n^{(1)}\left(\hat{\bm
			Y}\right)\right)^2}{M_n^{(2)}\left(\hat{\bm Y}\right)}
			- \frac{\left(M_n^{(1)}\left(\bm
			Y\right)\right)^2}{M_n^{(2)}\left(\bm Y\right)}\right|
			&= O_\mathbb{P}\left(\frac{h_n}{M_n^{(1)}\left(\bm Y\right)}
			\right), \label{eq:ab-error-rate} \\
			\left(1 - \frac{\left(M_n^{(1)}\left(\bm
			Y\right)\right)^2}{M_n^{(2)}\left(\bm Y\right)}\right)^{-1}
			&= O_\mathbb{P}\left(1\right),
			\quad\textnormal{and} \label{eq:a-bounded} \\
			\left(1 - \frac{\left(M_n^{(1)}\left(\hat{\bm
			Y}\right)\right)^2}{M_n^{(2)}\left(\hat{\bm Y}\right)}\right)^{-1}
			&= O_\mathbb{P}\left(1\right). \label{eq:b-bounded}
		\end{align}
		Additionally, by the continuous mapping theorem and~\cite[Lemma
		3.5.1]{dehaan2007}, we have
		\begin{equation} \label{eq:rate-hill}
			\left(M_n^{(1)}\left(\bm Y\right)\right)^{-1}
			= O_\mathbb{P}\left(\frac{U\left(\frac{n}{k}\right)}
			{a\left(\frac{n}{k}\right)}\right).
		\end{equation}
		Now, using
		\begin{equation*}
			\left(1 - x\right)^{-1} - \left(1 - y\right)^{-1}
			= \left(x - y\right)\left(1 - x\right)^{-1}\left(1 - y\right)^{-1},
			\quad x,y\in(0,1),
		\end{equation*}
		with $x = \frac{\left(M_n^{(1)}\left(\bm
		Y\right)\right)^2}{M_n^{(2)}\left(\bm Y\right)}$ and $y =
		\frac{\left(M_n^{(1)}\left(\hat{\bm
		Y}\right)\right)^2}{M_n^{(2)}\left(\hat{\bm Y}\right)}$, it follows
		from~\eqref{eq:ab-error-rate}--\eqref{eq:rate-hill} that
		\begin{equation} \label{eq:gamma-minus-error-rate}
			\left|\hat\gamma_-\left(\hat{\bm Y}\right)
			- \hat\gamma_-\left(\bm Y\right)\right|
			= O_\mathbb{P}\left(\frac{h_nU\left(\frac{n}{k}\right)}
			{a\left(\frac{n}{k}\right)}\right).
		\end{equation}
		By Lemma~\ref{lemma:dehaan}, we have
		\begin{equation*} 
			\lim_{n\to\infty}\frac{U\left(\frac{n}{k}\right)}
			{a\left(\frac{n}{k}\right)} =
			\begin{cases}
				\frac{1}{\gamma}, &\gamma > 0 \\
				\infty, &\gamma \leq 0
			\end{cases}.
		\end{equation*}
		Hence, combining~\eqref{eq:hill-error-rate}
		and~\eqref{eq:gamma-minus-error-rate} gives
		\begin{equation*}
			\left|\hat\gamma_M\left(\hat{\bm Y}\right)
			- \hat\gamma_M\left(\bm Y\right)\right|
			= O_\mathbb{P}\left(\frac{h_nU\left(\frac{n}{k}\right)}
			{a\left(\frac{n}{k}\right)}\right),
		\end{equation*}
		proving \eqref{eq:moment-error}. 
		
		Let us next prove~\eqref{eq:order-stat-error}. By recalling $\bm
		Y_{n-k,n} \stackrel{\mathcal{D}}{=}U\left(\bm Z_{n-k,n}\right)$, where
		$\bm Z$ is an i.i.d.\ sample from the standard Pareto distribution
		$F_Z\left(z\right) = \mathbbm{1}_{[1, \infty)}(z)\left(1 - 1/z\right)$,
		it follows from~\cite[Corollary 2.2.2 and Theorem B.2.18]{dehaan2007}
		that $\frac{\bm Y_{n-k,n} - U\left(\frac{n}{k}\right)}
		{a\left(\frac{n}{k}\right)}\stackrel{\mathbb{P}}{\to} 0$. Now 
		\begin{equation*}
			\left|\frac{\bm Y_{n-k,n}}{a\left(\frac{n}{k}\right)}\right|
			\leq \left|\frac{\bm Y_{n-k,n} - U\left(\frac{n}{k}\right)}
			{a\left(\frac{n}{k}\right)}\right|
			+ \left|\frac{U\left(\frac{n}{k}\right)}
			{a\left(\frac{n}{k}\right)}\right|
			= o_\mathbb{P}\left(1\right)
			+ \left|\frac{U\left(\frac{n}{k}\right)}
			{a\left(\frac{n}{k}\right)}\right|
			= O_\mathbb{P}\left(\frac{U\left(\frac{n}{k}\right)}
			{a\left(\frac{n}{k}\right)}\right),
		\end{equation*}
	 	and using \eqref{eq:error-cond} leads to
		\begin{equation*}
			\left|\frac{\hat{\bm Y}_{n-k,n} - \bm Y_{n-k,n}}
			{a\left(\frac{n}{k}\right)}\right|
			= \left|\frac{\hat{\bm Y}_{n-k,n}}{\bm Y_{n-k,n}} - 1\right|
			\left|\frac{\bm Y_{n-k,n}}{a\left(\frac{n}{k}\right)}\right|
			= O_\mathbb{P}\left(\frac{h_nU\left(\frac{n}{k}\right)}
			{a\left(\frac{n}{k}\right)}\right),
		\end{equation*}
		proving \eqref{eq:order-stat-error}.

		Let us next prove~\eqref{eq:a-error}. It follows from~\cite[Lemma
		3.5.1]{dehaan2007} that $1 - \hat\gamma_-\left(\bm Y\right) =
		O_\mathbb{P}\left(1\right)$ and that $M_n^{(1)}\left(\bm Y\right) =
		O_\mathbb{P}\left(\frac{a\left(\frac{n}{k}\right)}
		{U\left(\frac{n}{k}\right)}\right)$, which together
		with~\eqref{eq:hill-error-rate} gives 
		\begin{equation*}
			M_n^{(1)}\left(\hat{\bm Y}\right) =
			O_\mathbb{P}\left(\frac{a\left(\frac{n}{k}\right)}
			{U\left(\frac{n}{k}\right)}\right).
		\end{equation*}
		Together with \eqref{eq:gamma-minus-error-rate} and the assumption
		$\frac{h_nU\left(\frac{n}{k}\right)}{a\left(\frac{n}{k}\right)} = o(1)$,
		this gives
		\begin{equation*}
		\hat Z_n =
		O_\mathbb{P}\left(\frac{a\left(\frac{n}{k}\right)}
		{U\left(\frac{n}{k}\right)}\right)
		\end{equation*}
		for $\hat Z_n = M_n^{(1)}\left(\hat{\bm Y}\right) \left(1 -
		\hat\gamma_-\left(\hat{\bm Y}\right)\right)$. Let $Z_n =
		M_n^{(1)}\left(\bm Y\right) \left(1 - \hat\gamma_-\left(\bm
		Y\right)\right)$. Now 
		\begin{equation*}
			\begin{split}
				\left|\hat Z_n - Z_n\right|
				&\leq \left|M_n^{(1)}\left(\hat{\bm Y}\right)\right|
				\left|\hat\gamma_-\left(\hat{\bm Y}\right)
				- \hat\gamma_-\left(\bm Y\right)\right|
				+ \left|1 - \hat\gamma_-\left(\bm Y\right)\right|
				\left|M_n^{(1)}\left(\hat{\bm Y}\right)
				- M_n^{(1)}\left(\bm Y\right)\right| \\
				&= O_\mathbb{P}\left(\frac{a\left(\frac{n}{k}\right)}
				{U\left(\frac{n}{k}\right)}\right)
				O_\mathbb{P}\left(\frac{h_nU\left(\frac{n}{k}\right)}
				{a\left(\frac{n}{k}\right)}\right)
				+ O_\mathbb{P}\left(1\right)O_\mathbb{P}\left(h_n\right)
				= O_\mathbb{P}\left(h_n\right),
			\end{split}
		\end{equation*}
		and consequently,
		\begin{equation*}
			\begin{split}
				\left|\frac{\hat\sigma_M\left(\hat{\bm Y}\right) -
				\hat\sigma_M\left(\bm Y\right)}{a\left(\frac{n}{k}\right)}
				\right|
				&\leq \left|\hat Z_n\right|
				\left|\frac{\hat{\bm Y}_{n-k, n} - \bm Y_{n-k,n}}
				{a\left(\frac{n}{k}\right)}\right|
				+ \left|\frac{\bm Y_{n-k,n}}{a\left(\frac{n}{k}\right)}\right|
				\left|\hat Z_n - Z_n\right| \\
				&= O_\mathbb{P}\left(h_n\right)
				+ O_\mathbb{P}\left(\frac{h_nU\left(\frac{n}{k}\right)}
				{a\left(\frac{n}{k}\right)}\right)
				= O_\mathbb{P}\left(\frac{h_nU\left(\frac{n}{k}\right)}
				{a\left(\frac{n}{k}\right)}\right).
			\end{split}
		\end{equation*}
 		This proves~\eqref{eq:a-error}. 

		It remains to prove that~\eqref{eq:q-error} holds under the additional
		Assumptions~\ref{item:p-usual-1} and~\ref{item:gamma-bounded}. To
		simplify the notation, denote $\hat\gamma = \hat\gamma_M\left(\hat{\bm
		Y}\right)$, $\hat\sigma = \hat\sigma_M\left(\hat{\bm Y}\right)$,
		$\tilde\gamma = \hat\gamma_M\left(\bm Y\right)$ and $\tilde\sigma =
		\hat\sigma_M\left(\bm Y\right)$. We write
		\begin{equation*}
			\left|\hat x_p\left(\hat{\bm Y}\right) - \hat x_p\left(\bm Y\right)
			\right| \leq
			\underbrace{\left|\hat{\bm Y}_{n-k, n} - \bm Y_{n-k, n}\right|}
			_{\textnormal{I}}
			+ \underbrace{|\hat\sigma|\left| \frac{d_n^{\hat\gamma} - 1}
			{\hat\gamma} -  \frac{d_n^{\tilde\gamma} - 1}
			{\tilde\gamma}\right|}_{\textnormal{II}}
			+ \underbrace{|\hat\sigma - \tilde\sigma|\left|
			\frac{d_n^{\tilde\gamma} - 1}
			{\tilde\gamma} \right|}_{\textnormal{III}}
		\end{equation*}
		and bound the terms one by one.
		Using~\eqref{eq:order-stat-error} and~\eqref{eq:q} gives 
		\begin{equation*} 
			\textnormal{I} = O_\mathbb{P}\left(h_n U\left(\frac{n}{k}\right)
			\right)
			= O_\mathbb{P}\left(h_n U\left(\frac{n}{k}\right)
			q_\gamma\left(d_n\right)\right).
		\end{equation*}
		Consider next II. Using the bound $\left|\hat\sigma\right| \leq
		a\left(n/k\right)\left(\left|\frac{\hat\sigma -
		\tilde\sigma}{a\left(n/k\right)}\right| +
		\left|\frac{\tilde\sigma}{a\left(n/k\right)}\right|\right)$,
		\eqref{eq:a-error}, and~\cite[Theorem 4.2.1]{dehaan2007} gives
		$\left|\hat\sigma\right| = O_\mathbb{P}\left(a\left(\frac{n}{k}
		\right)\right)$ bounding the first factor in II. For the second factor
		in II, note first that \eqref{eq:moment-error} and
		\ref{item:gamma-bounded} imply $\sqrt{k}\left(\hat\gamma -
		\tilde\gamma\right) = O_\mathbb{P}\left(1\right)$. Moreover, by
		\ref{item:p-usual-1} we get 	$\left|\left(\hat\gamma -
		\tilde\gamma\right)\ln s\right| \leq \left|\sqrt{k}\left(\hat\gamma -
		\tilde\gamma\right)\right| \frac{\ln d_n}{\sqrt{k}} =
		o_\mathbb{P}\left(1\right)$ for any $1\leq s \leq d_n$. These imply 
		\begin{equation*}
			\sup_{1\leq s\leq d_n}
			\left|\frac{e^{\left(\hat\gamma - \tilde\gamma\right) \ln s} - 1}
			{\left(\hat\gamma - \tilde\gamma\right) \ln s}\right|
			\stackrel{\mathbb{P}}{\to} 1, \quad n\to\infty.
		\end{equation*}
		Together with
		$\frac{q_{\tilde\gamma}\left(d_n\right)}{q_\gamma\left(d_n\right)}
		\stackrel{\mathbb{P}}{\to} 1$ by \ref{item:p-usual-1}
		and~\ref{item:gamma-bounded} (see the proof of \cite[Corollary
		4.3.2]{dehaan2007}), this leads to
		\begin{equation*}
			\begin{split}
				\left|\frac{d_n^{\hat\gamma} - 1}{\hat\gamma}
				- \frac{d_n^{\tilde\gamma} - 1}{\tilde\gamma}\right|
				&= \left|\left(\hat\gamma - \tilde\gamma\right) \int_1^{d_n}
				s^{\tilde\gamma - 1}
				\frac{e^{\left(\hat\gamma - \tilde\gamma\right)\ln s} - 1}
				{\left(\hat\gamma - \tilde\gamma\right)\ln s}\ln s
				\,\mathrm{d}s\right| \\
				&\leq \left|\hat\gamma - \tilde\gamma\right|
				\sup_{1\leq s\leq d_n}
				\left|\frac{e^{\left(\hat\gamma - \tilde\gamma\right) \ln s}
					- 1}{\left(\hat\gamma - \tilde\gamma\right) \ln s}\right|
				q_{\tilde\gamma}\left(d_n\right) \\
				&= O_\mathbb{P}\left(\frac{h_nU\left(\frac{n}{k}\right)}
				{a\left(\frac{n}{k}\right)} q_\gamma\left(d_n\right)\right).
			\end{split}
		\end{equation*}
		Hence we have 
		\begin{equation*}
			\textnormal{II} = O_\mathbb{P}\left(h_n U\left(\frac{n}{k}\right)
			q_\gamma\left(d_n\right)\right).
		\end{equation*}
		Consider next III. As above, one can show that,
		under~\ref{item:p-usual-1} and~\ref{item:gamma-bounded}, we have 
			\begin{equation*}
			\left|\frac{d_n^{\tilde\gamma} - 1}{\tilde\gamma}
			- \frac{d_n^\gamma - 1}{\gamma}\right|
			= o_\mathbb{P}\left(q_\gamma\left(d_n\right)\right).
		\end{equation*}
		Now \eqref{eq:q} implies that $\left|\frac{d_n^\gamma -
		1}{\gamma}\right| = O\left(q_\gamma\left(d_n\right)\right)$, and using
		\eqref{eq:a-error} gives
			\begin{equation*} 
			\textnormal{III} \leq \left|\hat\sigma - \tilde\sigma\right|
			\left(\left|\frac{d_n^{\tilde\gamma} - 1}{\tilde\gamma}
			- \frac{d_n^\gamma - 1}{\gamma}\right|
			+ \left|\frac{d_n^\gamma - 1}{\gamma}\right|\right)
			= O_\mathbb{P}\left(h_n U\left(\frac{n}{k}\right)
			q_\gamma\left(d_n\right)\right).
		\end{equation*}
		This completes the proof.
	\end{proof}

	\begin{proof}[Proof of Lemma \ref{lemma:error-behavior}]
		By Part~\ref{item:gamma-pos-dehaan} of Lemma~\ref{lemma:dehaan} we have
		$\sqrt{k} z_n = O\left(\sqrt{k} h_n\right) = o\left(1\right)$, proving
		the case $\gamma>0$. Let next $\gamma=0$ and let $\delta > 0$. By
		Part~\ref{item:gamma-0-dehaan} of Lemma~\ref{lemma:dehaan}, $U$ is
		slowly varying, and   
		\cite[Theorem B.1.6]{dehaan2007} gives that $t^{-\delta}U(t)\to 0.$
		Additionally, Lemma~\ref{lemma:neves} gives  $a\left(t\right)\to c\in
		(0, \infty)$. Thus, $\sqrt{k} z_n = O\left(\sqrt{k} h_n
		U\left(n/k\right)\right) = O\left(\left(n/k\right)^{-\delta}
		U\left(n/k\right)\right) = o\left(1\right)$. Let now $\gamma<0$. By
		Part~\ref{item:gamma-neg-dehaan} of Lemma~\ref{lemma:dehaan},
		$U(\infty)<\infty$, and by Lemma~\ref{lemma:neves}, $\sqrt{k} z_n =
		O\left(\left(k/n\right)^\gamma \sqrt{k} h_n\right) = o\left(1\right)$.
		This completes the proof.
	\end{proof}

	\section{Proofs of Section \ref{sec:elliptical}}
	\label{appendix:sec:applications}
	In this section we present the proof of Theorem
	\ref{theorem:Q-hat-consistent}. For that, we give two auxiliary lemmas that
	are related to the effect of replacing true observations of the generating
	variate with approximations based on estimated location and scatter.

	Recall that $\|u\|_A = \sqrt{u^\intercal A^{-1} u}$. For the standard
	Euclidean norm, we use the notation $\|u\|$. For matrices, we use the
	induced norm defined by 
	\begin{equation*}
		\|B\| = \sup\left\{\frac{\|Bx\|}{\|x\|}
		: x\in\mathbb{R}^d\setminus\{0\}\right\},
		\quad B\in\mathbb{R}^{d\times d}.
	\end{equation*}
	We also recall the following basic inequalities 
	\begin{equation*}
		\|Bu\| \leq \|B\|\|u\|, \quad B\in\mathbb{R}^{d\times d},
		\ u\in\mathbb{R}^d,
	\end{equation*}
	and 
	\begin{equation*}
		\|BC\| \leq \|B\|\|C\|, \quad B,C\in\mathbb{R}^{d\times d}.
	\end{equation*}

	\begin{lemma} \label{lemma:gvariate-rate} Let $X$ be a $d$-variate
		elliptically distributed random variable with the location vector $\mu$,
		the scatter matrix $\Sigma$, and an absolutely continuous generating
		variate $\mathcal{R}$. Let $\bm X = \left(X_1, \ldots, X_n\right)$ be an
		i.i.d.\ sample from $X$ and suppose that the estimators $\hat\mu_n =
		\hat\mu\left(\bm X\right)$ and $\hat\Sigma_n = \hat\Sigma\left(\bm
		X\right)$ satisfy $\sqrt{n}\left(\hat\mu_n - \mu\right) =
		O_\mathbb{P}\left(1\right)$ and $\sqrt{n}\left(\hat\Sigma_n -
		\Sigma\right) = O_\mathbb{P}\left(1\right)$. Suppose further that
		$\hat\Sigma_n$ are symmetric positive definite and let $R$ and $\hat R$
		be as in~\eqref{eq:r}. Then
		\begin{equation*}
			\sqrt{n}\left(\hat R - R\right) = O_\mathbb{P}\left(1\right).
		\end{equation*}
	\end{lemma}

	\begin{proof}
		Since $\hat R\geq 0$ and $R > 0$ almost surely, we have $\left|\hat R -
		R\right| = \frac{\left|\hat R^2 - R^2\right|}{\left|\hat R + R\right|}
		\leq \frac{\left|\hat R^2 - R^2\right|}{R}$ almost surely as well. Thus
		it suffices to show that $\sqrt{n}\left|\hat R^2 - R^2\right| =
		O_\mathbb{P}\left(1\right)$. It follows from \eqref{eq:r} and triangle
		inequality that
		\begin{equation*}
			\left|\hat R^2 - R^2\right|
			\leq \underbrace{\left|\left\|X - \hat\mu_n\right\|_{\hat\Sigma_n}^2
			- \left\|X - \hat\mu_n\right\|_{\Sigma}^2\right|}_{\textnormal{I}}
			+ \underbrace{\left|\left\|X - \hat\mu_n\right\|_{\Sigma}^2
			- \left\|X - \mu\right\|_{\Sigma}^2\right|}_{\textnormal{II}}.
		\end{equation*}
		Now, consistency of $\hat\Sigma_n$ gives $\left\|\hat\Sigma_n^{-1} -
		\Sigma^{-1}\right\|\leq
		\left\|\hat\Sigma_n^{-1}\right\|\left\|\Sigma^{-1}\right\|
		\left\|\hat\Sigma_n - \Sigma\right\| =
		O_\mathbb{P}\left(\frac{1}{\sqrt{n}}\right)$, and consistency of
		$\hat\mu_n$ gives $\|X - \hat\mu_n\| = O_\mathbb{P}\left(1\right)$.
		Thus, by Cauchy--Schwarz inequality,
				\begin{equation*}
				\textnormal{I} = \left|\left(X - \hat\mu_n\right)^{\intercal}
				\left(\hat\Sigma_n^{-1} - \Sigma^{-1}\right)
				\left(X - \hat\mu_n\right)\right|
				\leq \left\|X-\hat\mu_n\right\|^2
				\left\|\hat\Sigma_n^{-1} - \Sigma^{-1}\right\|
				= O_\mathbb{P}\left(\frac{1}{\sqrt{n}}\right).
			\end{equation*}
		For Term II, reverse triangle inequality and equivalence of norms gives 
		\begin{equation*}
			\begin{split}
				\textnormal{II} &= \left|\left\|X - \hat\mu_n\right\|_\Sigma
				- \left\|X - \mu\right\|_\Sigma\right|
				\left|\left\|X - \hat\mu_n\right\|_\Sigma
				+ \left\|X - \mu\right\|_\Sigma\right| \\
				&\leq \left\|\hat\mu_n - \mu\right\|_\Sigma
				O_\mathbb{P}\left(1\right)
				= O_\mathbb{P}\left(\frac{1}{\sqrt{n}}\right).
			\end{split}
		\end{equation*}
		This completes the proof.
	\end{proof}

	The following lemma is a reformulation of~\cite[Lemma 2.2]{heikkila2019}.
	For the reader's convenience, we present the proof.
	\begin{lemma} \label{lemma:gvariate-error} Let $X$ be a $d$-variate
		elliptically distributed random variable with the location vector $\mu$,
		the scatter matrix $\Sigma$, and an absolutely continuous generating
		variate $\mathcal{R}$. Let $\bm X = \left(X_1, \ldots, X_n\right)$ be an
		i.i.d.\ sample from $X$ and suppose that the estimators $\hat\mu_n =
		\hat\mu\left(\bm X\right)$ and $\hat\Sigma_n = \hat\Sigma\left(\bm
		X\right)$ satisfy $\sqrt{n}\left(\hat\mu_n - \mu\right) =
		O_\mathbb{P}\left(1\right)$ and $\sqrt{n}\left(\hat\Sigma_n -
		\Sigma\right) = O_\mathbb{P}\left(1\right)$. Suppose further that
		$\hat\Sigma_n$ are symmetric positive definite and let $R_i$ and $\hat
		R_i$ be as in~\eqref{eq:ri}. Suppose $k=k_n$ satisfies
		Assumption~\ref{item:k}. Then 
			\begin{equation*}
			\max_{0\leq j\leq k}\left|\frac{\hat{\bm R}_{n-j,n}}{\bm R_{n-j,n}}
			- 1\right| = O_\mathbb{P}\left(\frac{1}{\sqrt{n}}\right).
		\end{equation*}
	\end{lemma}

	\begin{proof}
		Let $j\in\{0, \ldots, k\}$. By~\cite[Lemma 2.2]{heikkila2019}, there
		exists a sequence of nonnegative random variables $K_n =
		O_\mathbb{P}\left(\frac{1}{\sqrt{n}}\right)$ such that $\left|\hat{\bm
		R}_{n-j,n}^2 - \bm R_{n-j,n}^2\right|\leq K_n \bm R_{n-j, n}^2$. This
		gives  $K_n \geq \left|\frac{\hat{ \bm R}_{n-j,n}^2 - \bm
		R_{n-j,n}^2}{\bm R_{n-j,n}^2}\right| = \left|x_j^2 + 2x_j\right|$, where
		we have used the notation $x_j = \frac{\hat{\bm R}_{n-j,n}}{\bm
		R_{n-j,n}} - 1$. Since now $x_j\in [-1,\infty)$ almost surely, we have
		that $\max_{0\leq j\leq k} \left|x_j\right| = \max_{0\leq j\leq
		k}\left|\frac{x_j^2 + 2x_j}{x_j+2}\right|\leq K_n \max_{0\leq j\leq
		k}\left|\frac{1}{x_j+2}\right|\leq K_n =
		O_\mathbb{P}\left(\frac{1}{\sqrt{n}}\right)$. This completes the proof.
	\end{proof}

	\begin{proof}[Proof of Theorem \ref{theorem:Q-hat-consistent}]
		Throughout the proof we denote $U = U_\mathcal{R}$ and $a =
		a_\mathcal{R}$. Recall also the notations $d_n = k/(np)$,
		$q_\gamma\left(t\right) = \int_1^t s^{\gamma - 1}\ln s \, \mathrm{d}s$,
		$r_p = U\left(1/p\right)$, and
		\begin{equation*}
			\hat Q_p = \left\{x\in\mathbb{R}^d :
			\left\|x - \hat\mu\right\|_{\hat\Sigma}
			\geq \hat x_p\left(\hat R\right)\right\}.
		\end{equation*}
		Since $\mathcal{R}$ satisfies \ref{item:density} and $p = p_n\to 0$, we
		have the representation (see Section~\ref{sec:elliptical-preliminaries})
		\begin{equation*}
			Q_p = \left\{x\in\mathbb{R}^d : \left\|x - \mu\right\|_\Sigma
			\geq r_p\right\}
		\end{equation*}
		for the quantile regions defined in~\eqref{eq:region}. By
		Lemma~\ref{lemma:error-behavior} we have that
		\begin{equation} \label{eq:elliptical-error}
			\sqrt{k_n}z_n = \sqrt{\frac{k}{n}} \frac{U\left(\frac{n}{k}\right)}
			{a\left(\frac{n}{k}\right)}\to 0.
		\end{equation}
		Now, using Lemma~\ref{lemma:gvariate-error}, we see that we can apply
		Theorem~\ref{theorem:moment} with $h_n = \frac{1}{\sqrt{n}}$, $\bm Y =
		\bm R$, and $\hat{\bm Y} = \hat{\bm R}$. Furthermore, since
		$\mathcal{R}$ satisfies~\ref{item:2erv}, we infer that
		\begin{equation} \label{eq:r-hat-normality}
			\sqrt{k}\left(\hat\gamma_M\left(\hat{\bm R}\right) - \gamma,
			\frac{\sigma_M\left(\hat{\bm R}\right)}{a\left(\frac{n}{k}\right)}
			- 1,
			\frac{\hat{\bm R}_{n-k,n}
				- U\left(\frac{n}{k}\right)}{a\left(\frac{n}{k}\right)}\right)
			\stackrel{\mathcal{D}}{\to} \left(\Gamma, \Lambda, B\right),
		\end{equation}
		where $\left(\Gamma, \Lambda, B\right)$ is jointly normal random vector
		with a known mean vector and a known covariance matrix (cf.
		Remark~\ref{remark:normality}). Since~\ref{item:p-usual-2}
		implies~\ref{item:p-usual-1}, these derivations lead to
		\begin{equation} \label{eq:q-hat-normality}
			\sqrt{k}\frac{\hat x_p\left(\hat{\bm R}\right) - r_p}
			{a\left(\frac{n}{k}\right)q_\gamma\left(d_n\right)}
			= O_\mathbb{P}\left(1\right).
		\end{equation}

	 	Let 
		\begin{equation*}
			\tilde Q_p = \left\{x\in\mathbb{R}^d : \left\|x - \mu\right\|
			_\Sigma \geq \max\left\{r_p, \tilde r_p\right\}\right\},
		\end{equation*}
		where $\tilde r_p = \left\|\hat\mu\left(\bm X\right) -
		\mu\right\|_\Sigma + \hat x_p\left(\hat{\bm
		R}\right)\left(\left\|\hat\Sigma\left(\bm
		X\right)\right\|\left\|\Sigma^{-1} - \hat\Sigma^{-1}\left(\bm
		X\right)\right\| + 1\right)$. Recall that $\mathbb{P}\left(X\in
		A\triangle B\right)$ is a pseudometric. By~\cite[Lemma 2]{pere2024a} we
		have $\tilde Q_p\subset \hat Q_p\cap Q_p$, and thus,
		\begin{equation*}
			\begin{split}
				\frac{\mathbb{P}\left(X\in\hat Q_p\triangle Q_p\right)}{p}
				&\leq \frac{\mathbb{P}\left(X\in \hat Q_p\triangle \tilde Q_p
				\right)}{p}
				+ \frac{\mathbb{P}\left(X\in\tilde Q_p\triangle Q_p\right)}
				{p} \\
				&= \left(\frac{\mathbb{P}\left(X\in\hat Q_p\right)}{p}
				- \frac{\mathbb{P}\left(X\in\tilde Q_p\right)}{p}\right) +
				\left(\frac{\mathbb{P}\left(X\in Q_p\right)}{p}
				- \frac{\mathbb{P}\left(X\in\tilde Q_p\right)}{p}\right) \\
				&= 1 + \frac{\mathbb{P}\left(X\in\hat Q_p\right)}{p}
				- 2\frac{\mathbb{P}\left(X\in\tilde Q_p\right)}{p}.
			\end{split}
		\end{equation*}
		Thus it suffices to prove that, as $n\to\infty$,
		\begin{align}
			\frac{\mathbb{P}\left(X\in\hat Q_p\right)}{p}
			&\stackrel{\mathbb{P}}{\to}1 \quad\textnormal{and}
			\label{eq:qhat} \\
			\frac{\mathbb{P}\left(X\in\tilde Q_p\right)}{p}
			&\stackrel{\mathbb{P}}{\to}1. \label{eq:qtilde}
		\end{align}

		We begin by showing~\eqref{eq:qhat}. We have 
		\begin{equation*}
			\frac{\mathbb{P}\left(X\in\hat Q_p\right)}{p}
			= \frac{\mathbb{P}\left(\mathcal{R} > \hat x_p\left(\hat
			R\right)+ \left(R - \hat R\right) \right)}{p},
		\end{equation*}
		where
		\begin{equation*}
			\hat x_p\left(\hat R\right)+ \left(R - \hat R\right)
			= \underbrace{\hat{\bm R}_{n-k,n} + \left(R - \hat R\right)}
			_{\hat b_n}
			+ \hat \sigma_M\left(\hat{\bm R}\right) \frac{\left(\frac{k}{np}
				\right)^{\hat\gamma_M\left(\hat{\bm R}\right)} - 1}
			{\hat \gamma_M\left(\hat{\bm R}\right)}.
		\end{equation*}
		By~\cite[Theorem A.1.]{einmahl2009} we have
		$\frac{\mathbb{P}\left(\mathcal{R} > \hat x_p\left(\hat R\right)+
		\left(R - \hat R\right) \right)}{p}\stackrel{\mathbb{P}}{\to} 1$,
		provided that the quantities $\sqrt{k}\left(\frac{\hat b_n -
		U\left(\frac{n}{k}\right)}{a\left(\frac{n}{k}\right)}\right)$,
		$\sqrt{k}\left(\hat\gamma_M\left(\hat{\bm R}\right) - \gamma\right)$ and
		$\sqrt{k}\left(\frac{\hat\sigma_M\left(\hat{\bm
		R}\right)}{a\left(\frac{n}{k}\right)} - 1\right)$ are bounded in
		probability. Since the latter two follow directly
		from~\eqref{eq:r-hat-normality}, it suffices to show
		$\sqrt{k}\left(\frac{\hat b_n -
		U\left(\frac{n}{k}\right)}{a\left(\frac{n}{k}\right)}\right)
		=O_\mathbb{P}\left(1\right)$. For this, note that since, again
		by~\eqref{eq:r-hat-normality}, $\sqrt{k}\left(\frac{\hat{\bm R}_{n-k,n}
		- U\left(\frac{n}{k}\right)}{a\left(\frac{n}{k}\right)}\right)
		=O_\mathbb{P}\left(1\right)$, it remains to show that
		$\frac{\sqrt{k}\left(\hat R - R\right)}{a\left(\frac{n}{k}\right)} =
		O_\mathbb{P}\left(1\right)$. Now, Lemma~\ref{lemma:gvariate-rate},
		together with \eqref{eq:elliptical-error}, implies
		\begin{equation*}
			\frac{\sqrt{k}\left(\hat R - R\right)}{a\left(\frac{n}{k}\right)}
			= O_\mathbb{P}\left(\sqrt{\frac{k}{n}}
			\frac{U\left(\frac{n}{k}\right)}
			{a\left(\frac{n}{k}\right)}\right)
			\frac{1}{U\left(\frac{n}{k}\right)}
			\stackrel{\mathbb{P}}{\to} 0.
		\end{equation*}
		Thus we have obtained~\eqref{eq:qhat}. 
		
		We next show~\eqref{eq:qtilde}. Using $\min\left\{x, y\right\} =
		\frac{x+y-\left|x - y\right|}{2}$ gives 
		\begin{equation*}
			\begin{split}
				\frac{\mathbb{P}\left(X\in\hat Q_p\right)}{p}
				&= \frac{\mathbb{P}\left(\mathcal{R} > \max\left\{r_p,
				\tilde r_p\right\}\right)}{p}
				= \frac{\min\left\{\mathbb{P}\left(\mathcal{R} > r_p\right),
				\mathbb{P}\left(\mathcal{R} > \tilde r_p\right)\right\}}{p} \\
				&= \frac{1}{2}\left(1 + \frac{\mathbb{P}\left(\mathcal{R}
				> \tilde r_p\right)}{p}\right)
				- \frac{1}{2}\left|1 - \frac{\mathbb{P}\left(\mathcal{R}
				> \tilde r_p\right)}{p}\right|,
			\end{split}
		\end{equation*}
		and thus, it suffices to show $\frac{\mathbb{P}\left(\mathcal{R} >
		\tilde r_p\right)}{p} \stackrel{\mathbb{P}}{\to} 1$. We write 
		\begin{equation*}
			\tilde r_p =
			\underbrace{\left\|\hat\mu\left(\bm X\right) - \mu\right\|_\Sigma
			+ \left\|\hat\Sigma\left(\bm X\right)\right\|\left\|\Sigma^{-1}
			- \hat\Sigma\left(\bm X\right)^{-1}\right\|
			\hat x_p\left(\hat{\bm R}\right)}_{\tilde b_n}
			+ \hat{\bm R}_{n-k,n}
			+ \hat\sigma_M\left(\hat{\bm R}\right)
			\frac{\left(\frac{k}{np}\right)
			^{\hat\gamma_M\left(\hat{\bm R}\right)} - 1}
			{\hat \gamma_M\left(\hat{\bm R}\right)}.
		\end{equation*}
		Now, similarly as in the proof of~\eqref{eq:qhat},
		$\frac{\mathbb{P}\left(\mathcal{R} > \tilde r_p\right)}{p}
		\stackrel{\mathbb{P}}{\to} 1$ follows from~\eqref{eq:r-hat-normality}
		and~\cite[Theorem A.1.]{einmahl2009} provided that $\frac{\sqrt{k}\tilde
		b_n}{a\left(\frac{n}{k}\right)} = O_\mathbb{P}\left(1\right)$.
		Using~\ref{item:loc-scat} and~\eqref{eq:q-hat-normality} gives
		\begin{equation*}
			\begin{split}
				\frac{\sqrt{k}\tilde b_n}{a\left(\frac{n}{k}\right)}
				&= O_\mathbb{P}\left(\frac{\sqrt{k}}
				{\sqrt{n}a\left(\frac{n}{k}\right)}\right)
				+ O_\mathbb{P}\left(\frac{\sqrt{k}}{\sqrt{n}
					a\left(\frac{n}{k}\right)}\right)
				\left(\left(\hat x_p\left(\hat{\bm R}\right)
				- r_p\right)+ r_p\right) \\
				&= O_\mathbb{P}\left(\frac{\sqrt{k}}{\sqrt{n}
					a\left(\frac{n}{k}\right)}\right)
				+ O_\mathbb{P}\left(\frac{
					q_\gamma\left(d_n\right)}{\sqrt{n}}\right)
				+ O_\mathbb{P}\left(\frac{\sqrt{k}r_p}
				{\sqrt{n}a\left(\frac{n}{k}\right)}\right).
			\end{split}
		\end{equation*}
		Since $U\left(\infty\right) \in (0, \infty]$,
		Equation~\eqref{eq:elliptical-error} gives
		$\frac{\sqrt{k}}{\sqrt{n}a\left(\frac{n}{k}\right)} =
		O\left(\sqrt{k}z_n\right) = o\left(1\right)$ for the first term.
		Moreover, in the case $\gamma \leq 0$, we have $d_n^\gamma =
		O\left(1\right)$, and in the case $\gamma>0$, \ref{item:p-slow} gives
		$\sqrt{\frac{k}{n}}d_n^\gamma = O\left(1\right)$. Thus,
		by~\ref{item:p-usual-2}, this leads to  $\frac{
		q_\gamma\left(d_n\right)}{\sqrt{n}} =o(1)$. We consider the last term
		$\frac{\sqrt{k}r_p} {\sqrt{n}a\left(\frac{n}{k}\right)}$ separately for
		$-1/2 < \gamma < 0$, $\gamma = 0$ and $\gamma > 0$. For $-1/2 < \gamma <
		0$, by using $U\left(\infty\right)\in(0, \infty)$
		and~\eqref{eq:elliptical-error}, we obtain $\frac{\sqrt{k}r_p}
		{\sqrt{n}a\left(\frac{n}{k}\right)}=
		O\left(\sqrt{\frac{k}{n}}\frac{1}{a\left(n/k\right)}\right) =
		O\left(\sqrt{k} z_n\right) = o\left(1\right)$. For $\gamma>0$,
		Lemma~\ref{lemma:dehaan} gives $U\left(t\right) / a\left(t\right) =
		O\left(1\right)$, and Lemma~\ref{lemma:neves} and~\ref{item:p-slow} now
		provide
		\begin{equation*}
			\frac{\sqrt{k}r_p}
			{\sqrt{n}a\left(\frac{n}{k}\right)} = \sqrt{\frac{k}{n}}
			\frac{U\left(1/p\right)}{a\left(1/p\right)}
			\frac{a\left(1/p\right)}{\left(1/p\right)^\gamma}
			\frac{\left(n/k\right)^\gamma}{a\left(n/k\right)}
			\left(\frac{k}{np}\right)^\gamma
			= O\left(\sqrt{\frac{k}{n}} d_n^\gamma\right) = O(1).
		\end{equation*}
		Finally, for $\gamma=0$, we write
		\begin{equation*}
			r_p = U\left(\frac{1}{p}\right) = a\left(\frac{n}{k}\right)
			\left(\frac{U\left(\frac{n}{k}d_n\right)
			- U\left(\frac{n}{k}\right)}{a\left(\frac{n}{k}\right)}
			\frac{1}{\ln d_n}\ln d_n
			+ \frac{U\left(\frac{n}{k}\right)}{a\left(\frac{n}{k}\right)}
			\right).
		\end{equation*}
		By~\cite[Lemma 4.3.5]{dehaan2007} we have
		$\frac{U\left(\frac{n}{k}d_n\right) -
		U\left(\frac{n}{k}\right)}{a\left(\frac{n}{k}\right)} \frac{1}{\ln d_n}
		\to 1$, and consequently, $r_p = O\left(a\left(\frac{n}{k}\right) \ln
		d_n\right) + U\left(\frac{n}{k}\right)$. This gives
		\begin{equation*}
			\frac{\sqrt{k}r_p}{\sqrt{n}a\left(\frac{n}{k}\right)}
			= O\left(\sqrt{\frac{k}{n}}\ln d_n\right)
			+ \sqrt{k}z_n.
		\end{equation*}
		Hence \eqref{eq:qtilde} follows from~\ref{item:p-slow}
		and~\eqref{eq:elliptical-error}. This completes the whole proof.  
	\end{proof}
\end{appendices}
\end{document}